\documentclass{birkjour}
\usepackage[utf8]{inputenc}
\usepackage{amsfonts,amsmath,amsthm,amssymb}
\usepackage{graphics, epsfig}
\usepackage{color}
\usepackage{appendix}
\usepackage{ulem}
\usepackage[makeroom]{cancel}
\usepackage[usenames,dvipsnames]{pstricks}
\usepackage{pst-grad} 
\usepackage{pst-plot} 

\newtheorem{theorem}{Theorem}[section]
\newtheorem{lemma}{Lemma}[section]

\newtheorem{remark}{Remark}[section]
\renewcommand{\thesection}{\arabic{section}}
\renewcommand{\theequation}{\thesection.\arabic{equation}}

\numberwithin{equation}{section}
\numberwithin{theorem}{section}
\numberwithin{proposition}{section}
\numberwithin{lemma}{section}
\numberwithin{remark}{section}
\newcommand{\dist}{\operatorname{dist}}

\newcommand{\intl}{\int\limits}

\def\Xint#1{\mathchoice
    {\XXint\displaystyle\textstyle{#1}}%
    {\XXint\textstyle\scriptstyle{#1}}%
    {\XXint\scriptstyle\scriptscriptstyle{#1}}%
    {\XXint\scriptscriptstyle\scriptscriptstyle{#1}}%
    \!\int}
\def\XXint#1#2#3{\setbox0=\hbox{$#1{#2#3}{\int}$}
    \vcenter{\hbox{$#2#3$}}\kern-0.5\wd0}
\def\bint{\Xint-}
\def\dashint{\Xint{\raise4pt\hbox to7pt{\hrulefill}}}
\def\dashiint{\bint\kern-0.15cm\bint}

\newcommand{\ovl}[3]{\int_{#1}^{#2}\kern-#3pt\raise4pt\hbox to7pt{\hrulefill}\ }

\newcommand{\ovll}[3]{\intl_{#1}^{#2}\kern-#3pt\raise4pt\hbox to7pt{\hrulefill}\ }

\newcommand{\tvl}[2]{\iint_{#1}\kern-#2pt\raise4pt\hbox to7pt{\hrulefill}\ }

\newcommand{\bye}{\end{document}}

\newcommand{\tvls}[2]{\iint_{#1}\kern-#2pt\raise4pt\hbox to15pt{\hrulefill}\ }

\newtheorem{Definition}{Definition}[section]
\begin{document}
\title[Positive Eigenfunctions via Sub-Super Solutions Method]{Existence of Positive Eigenfunctions to an Anisotropic Elliptic Operator via Sub-Super Solutions Method}

\author[S. Ciani]{Simone Ciani}
\address{Dpto. di Matematica e Informatica ``U. Dini", \\
Universit\`a degli Studi di Firenze,\\ 
viale G. Morgagni 67/A, 50134 Firenze, Italy}
\email{\tt simone.ciani@unifi.it}

\author[G. M. Figueiredo]{Giovany M. Figueiredo}
\address{Dpto. de Matemática,\\
Universidade de Brasilia,\\
UNB, CEP: 70910-900, Brasília–DF, Brazil }
\email{\tt giovany@unb.br}

\author[A. Su\'arez]{Antonio Su\'arez}
\address{Dpto. EDAN and IMUS,\\
University of Sevilla,\\
Avda. Reina Mercedes, s/n, 41012,
Sevilla, Spain}
\email{\tt suarez@us.es}

\subjclass{35K65, 35B65, 35B45, 35K20.}

\keywords{Anisotropic $p$-Laplacian, Positive Solution, Sub-Supersolution, Eigenvalues.}

\date{\today}

\begin{abstract}
Using the sub-supersolution method we study the existence of positive solutions for the anisotropic problem
\begin{equation} \label{0.1}
-\sum_{i=1}^N\frac{\partial}{\partial x_i}\left( \left|\frac{\partial u}{\partial x_i}\right|^{p_i-2}\frac{\partial u}{\partial x_i}\right)=\lambda u^{q-1} 
\end{equation} 
where $\Omega$ is a bounded and regular domain of $\mathbb{R}^N$, $q>1$ and $\lambda>0$.
\end{abstract}

\maketitle


\section{Introduction}\label{S:intro}

In this paper the main goal is to show the existence of positive solutions of the problem 
\begin{equation}
\label{intro}
\left\{\begin{array}{ll}
-\displaystyle\sum_{i=1}^N\frac{\partial}{\partial x_i}\left( \left|\frac{\partial u}{\partial x_i}\right|^{p_i-2}\frac{\partial u}{\partial x_i}\right)=\lambda u^{q-1}  &\mbox{in $\Omega$,}\\
u=0 & \mbox{on $\partial\Omega$,}
\end{array}
\right.
\end{equation} 
where $\Omega\subset \mathbb{R}^N$, $N\geq 1$, is a bounded and regular domain, $p_i>1$, $i=1,\ldots, N$, $q>1$ and $\lambda$ is a real parameter. We will assume without loss of generality that the $ p_{i}$ are ordered increasingly, that is, $p_1<\ldots <p_N$.\newline
There is a vast literature concerning to anisotropic elliptic problems. We mention here only those references most strongly related to (\ref{intro}).
First, in \cite{FGK} it was proved that for $q<p_N$ for any $\gamma>0$ there exists $\lambda_\gamma>0$ and $u_\gamma$ with $\|u_\gamma\|_p=\gamma$ and $u_\gamma$ solution of (\ref{intro}) with $\lambda=\lambda_\gamma$. As the authors themselves claim, from this result it can not be deduced the existence of solutions of (\ref{intro}) for a given $\lambda$. In \cite{Monte}, using mainly variational methods, it was proved that if $p_1<q<p_N$ then there exist $0<\lambda_* \leq \lambda^*$ such that:
\begin{itemize}
\item If $\lambda\leq \lambda_*$, (\ref{intro}) does not posses positive solution.
\item If $\lambda> \lambda^*$, (\ref{intro}) possesses at least a positive solution.
\end{itemize}
Finally, for the general results of \cite{MPR} (Corollary 1) we can deduce that for the case $1<q<p_1$ there exist $0<\lambda_*<\lambda_{**}$ such that (\ref{intro}) possesses at least a solution for $\lambda\in (0,\lambda_*)\cup (\lambda_{**},\infty)$.\newline
In this paper we complete and improve the above results. For that, we use the sub-supersolution method, see \cite{ElHam}, \cite{giosusbsup} and \cite{Struwe}, (see also \cite{GST}, \cite{Giovany-Silva1}, \cite{Giovany-Silva2} and references therein for the application of this method to problems with nonlinear reaction function including singularities or critical exponent).\newline
This method allows us not only to prove the existence of a solution, but also gives us lower and upper bounds of such solution. Specifically, our main result is the following.
\begin{theorem} \quad
\label{main}
\begin{enumerate}
\item Assume that $1<q<p_1$. There exists a positive solution of (\ref{intro}) if and  only if $\lambda>0$.
\item Assume that $p_1\leq q < p_N$. There exists $\Lambda>0$ such that (\ref{intro}) does not posses positive solutions for $\lambda<\Lambda$ and (\ref{intro}) possesses at least one positive solution for $\lambda>\Lambda$. 
\end{enumerate}
\end{theorem} \noindent
An outline of the paper is as follows: in Section 2 we recall some definitions and some properties of the eigenvalues and eigenfunctions of the classical $p$-Laplacian. Next in Section 3 we enunciate the sub-supersolution method. Then in Section 4 we construct sub and super-solutions by multiplication of powers of $p$-Laplacian eigenfunctions to be applied in the existence theorem.
\section{Preliminary Lemmas and Setting}  
Consider  $h (x,s): \Omega \times \mathbb{R} \rightarrow \mathbb{R}$  a Caratheodory function, i.e. measurable in $x$ and continuous in the second variable $s$. Consider the anisotropic problem
\begin{equation} 
\label{proto}
\left\{\begin{array}{ll}
-\displaystyle\sum_{i=1}^N\frac{\partial}{\partial x_i}\left( \left|\frac{\partial u}{\partial x_i}\right|^{p_i-2}\frac{\partial u}{\partial x_i}\right)=h(x,u(x))  &  \text{in  $\Omega$},\\
u=0  & \text{on $\partial \Omega.$}
\end{array}
\right.
\end{equation}
The natural framework to study (\ref{proto}) is the anisotropic Sobolev Space $W^{1,\bf{p}}_{0}(\Omega)$, that is, the closure of $C_{0}^{\infty}(\Omega)$ under the anisotropic norm
$$ 
\|u \|_{W^{1,\bf{p}}(\Omega)}:=  \sum_{i=1}^{N} \bigg{\|} \frac{\partial u}{\partial x_i} \bigg{\|}_{p_i} 
$$
where $\frac{\partial u}{\partial x_i}$ denotes the $i-$th weak partial derivative of $u$.\newline
Recall that if we denote 
\begin{equation} 
\label{picondition}
\sum_{i=1}^{N} \frac{1}{p_{i}} > 1,  \quad p_{i}>1  \quad \forall i=1,\ldots,N, \quad p^{*}:= \frac{N}{\sum \frac{1}{p_{i}}-1}, \quad p_{\infty}:=\max \{p^{*}, p_{N} \},
\end{equation} \noindent 
then for every $r \in [1, p_{\infty}]$ the embedding $$W^{1,\bf{p}}_{0}(\Omega) \subset L^{r}(\Omega)$$ is continuous, and it is compact if $r<p_{\infty}$.
More precisely, it holds the following directional Poincar\'e-type inequality for any $u\in C^1_c(\Omega)$ (see for instance \cite{FGK}) 
\begin{equation} 
\label{embedding}
||u||_{r} \leq \frac{d^i\, r}{2} \bigg| \bigg| \frac{\partial u}{\partial x_i} \bigg| \bigg|_{r}, \quad \forall r\geq 1, \quad d^i= \sup_{x,y \in \Omega} \langle x-y,e_i\rangle,
\end{equation}
denoting by $\{e_1,\ldots,e_N\}$ the canonical basis of $\mathbb{R}^N$. \newline
The theory of embeddings of this kind of anisotropic Sobolev spaces is vast and we refer to \cite{FGK} for directional Poincar\'e-type inequality and to \cite{Schm} for Sobolev and Morrey's embeddings of the whole $W^{1,{\bf{p}}}(\Omega)$ space, obtained with an important geometric condition on the domain $\Omega$, namely that it must be semi-rectangular. It is not a case that this semi-rectangular condition reflects in our construction of the solution: the existence of traces for this kind of functions is heavily depending on the geometry of the domain as shown in \cite{Schm}. Regularity theory for orthotropic operators as the one defined by equation \eqref{intro} is still a challenging open problem, see for example \cite{DiBe}. \newline We also recall the following definition.
\begin{Definition}
A function $u \in W^{1,\bf{p}}(\Omega)$ is defined to be a a sub-(super-) solution to the problem (\ref{proto}) if $u \le(\ge) \, 0$ in $\partial \Omega$ and $\forall \hspace{0,1 cm} 0 \le \phi \in W^{1,{\bf{p}}}_0(\Omega)$ it satisfies
\begin{equation} \label{subsolution}
\int_{\Omega} \bigg[ \displaystyle\sum_{i=1}^N\left|\frac{\partial u}{\partial x_i}\right|^{p_i-2}\frac{\partial u}{\partial x_i}\frac{\partial \phi}{\partial x_i}  - h(x,u(x)) \phi  \bigg] dx \leq (\ge) 0.
\end{equation}
Finally, a solution $u\in W^{1,\bf{p}}_{0}(\Omega)$ to (\ref{proto}) has to satisfy 
$$
\int_{\Omega} \bigg[ \displaystyle\sum_{i=1}^N\left|\frac{\partial u}{\partial x_i}\right|^{p_i-2}\frac{\partial u}{\partial x_i}\frac{\partial \phi}{\partial x_i}  - h(x,u(x)) \phi  \bigg] dx = 0\quad \forall \phi \in W^{1,\bf{p}}_0(\Omega).
$$
\end{Definition} \noindent
Now, we recall some well-known results concerning the eigenvalue problem for the $p$-Laplacian. Specifically, the problem
\begin{equation}
\label{eigen} 
\left\{\begin{array}{ll} 
- \Delta_{p} u= \lambda |u|^{p-2} u &\text{in $\Omega$,}\\
u=0 & \text{on $\partial \Omega$,} 
\end{array}
\right.
\end{equation}
where 
$$
\Delta_{p} u= \text{div} (|\nabla u|^{p-2}\nabla u)=\displaystyle\sum_{i=1}^N\frac{\partial}{\partial x_i}\left( \left|\frac{\partial u}{\partial x_i}\right|^{p-2}\frac{\partial u}{\partial x_i}\right)
$$ 
The following result is well-known:
\begin{lemma}
\label{autovalor}
The eigenvalue problem (\ref{eigen}) has a unique eigenvalue $\lambda=\lambda_1$ with the property of having a positive associated eigenfunction $\varphi_1\in W_0^{1,p}(\Omega)$, called principal eigenfunction. Moreover, $\lambda_1$ is simple, isolated and is defined by
$$
\lambda_1=\inf\left\{\int_\Omega|\nabla u|^p: u\in W^{1,p}_0(\Omega),\;\int_\Omega |u|^p dx=1\right\}.
$$
 Furthermore, $\varphi_1\in C^{1,\beta}(\overline\Omega)$ for some $\beta\in (0,1)$ and $\partial\varphi_1/\partial n<0$ on $\partial\Omega$, where $n$ is the outward unit normal on $\partial\Omega$. Finally, for $N=1$ we have that 
\begin{equation}
\label{cosa}
|\nabla \varphi_1|^{p-2}\nabla \varphi_1\in W^{1,2}(\Omega),
\end{equation}
and in fact
$$
-\Delta_p \varphi_1(x)=\lambda_1\varphi_1(x)\quad \mbox{a.e. $x\in \Omega$}.
$$
\end{lemma}
\begin{remark}
The existence of $\lambda_1$ and main properties of $\varphi_1$ are well-known, see \cite{Giusti}, \cite{Peter1}, \cite{Peral}. Property (\ref{cosa}) holds in $N=1$, see for instance \cite{gv}, and for $N\geq 2$ is some specific domains, for example for $\Omega$ convex, see \cite{Cia}.  
\end{remark}

\section{An existence sub-supersolution theorem}
We start by stating an important theorem, see \cite{ElHam} and \cite{giosusbsup}, that assures the existence of a solution between a sub and a supersolution.
\begin{theorem} 
\label{teoss}
Suppose that $h: \mathbb{R}\mapsto \mathbb{R}$ a continuous function and that there exist $\underline{u}, \overline{u} \in W^{1,\bf{p}}(\Omega)\cap L^\infty(\Omega)$ subsolution and supersolution of (\ref{proto}) such that $\underline{u}\leq \overline{u}$. Then there exists $u \in W^{1,\bf{p}}_{0}(\Omega)$ solution to (\ref{proto}) such that 
$$
\underline{u} \le u \le \overline{u}.
$$
\end{theorem}
\begin{proof}
Since $\underline{u}$ and $\overline{u}$ belong to $L^\infty(\Omega)$, then $h$ verifies condition $(h_2)$ of \cite{ElHam}. This concludes the proof.
\end{proof}

\section{Construction of sub and super-solutions: proof of the main result}
In this section we prove Theorem \ref{main}. For that, we apply Theorem \ref{teoss} to (\ref{intro}). Mainly, we construct the sub and the supersolution. 

\subsection{Sub-solutions}
Let us consider a rectangular bounded domain $U \subseteq \Omega$ i.e. 
$$
U:= \prod_{i=1}^{N}U_i,\quad\mbox{where $U_{i}=(a_i,b_i)$,} \quad a_i,b_i \in \mathbb{R} \quad \forall \, i=1,..,N.
$$
Denote by  $v_i=v_i(x_i)$ a positive principal eigenfunction of  $-\Delta_{p_i}$ in $U_i$,\newline that is, 
\begin{equation}
\label{pi} 
\begin{cases} - \Delta_{p_{i}} v_i= \eta_{i} |v_i|^{p_{i}-2} v_i &\text{in $U_i$,}\\
v_i=0 &\text{on  $\partial U_i$.} 
\end{cases}
\end{equation} 
From Lemma \ref{autovalor}, recall that, if $n_i$ is the outward normal derivative to $\partial U_i$, we have
\begin{equation}
\label{normal}
\frac{\partial v_i}{\partial n_i}<0 \quad\mbox{on $\partial U_i$.}
\end{equation}
Let us consider the function
\begin{equation}
    \underline{u}(x)= \begin{cases}
     \epsilon \, \displaystyle\prod_{i=1}^{N}  v_i^{\alpha_i} (x_{i}) & x \in U, \\
     0 & x\in \Omega\setminus\overline{U},
    \end{cases}
\end{equation}
where $\alpha_i>0$, $i=1,\ldots,N$, and $\epsilon>0$ will be chosen later. 
\begin{remark}
We note that $\underline{u}(x)>0 $ in $\emptyset \ne U \subset \Omega$.
\end{remark} \noindent
As $v_i$ are bounded, it is clear that  $\underline{u} \in W^{1,{\bf{p}}} (\Omega)$ and that $\underline{u}_{|\partial \Omega} = 0 $. Hence, $\underline{u}$ is subsolution  of (\ref{intro}) provided that 
$$
\int_{\Omega} \sum_{i=1}^{N} \bigg| \frac{\partial \underline{u}}{\partial x_i} \bigg|^{p_{i}-2}  \frac{\partial \underline{u}}{\partial x_i} \, \frac{\partial \phi}{\partial x_i} \, dx \leq \lambda \int_{\Omega} \underline{u}^{q-1} \phi \, dx  \quad\forall \, \phi \in W^{1,{\bf{p}}}_{0} (\Omega), \quad \phi \ge 0.
$$ \noindent
Observe that
\begin{equation} 
\lambda \int_{\Omega} \underline{u}^{q-1} \phi \, dx = \lambda \epsilon^{q-1}  \int_{U} \prod_{i=1}^{N}v_{i}^{\alpha_{i} (q-1)} \phi \, dx.
\end{equation} \noindent
On the other hand, observe that
$$
 \frac{\partial \underline{u}}{\partial x_i} = \epsilon \,  \alpha_{i} \bigg( \prod_{j\ne i} v_j^{\alpha_j} \bigg) v_i^{\alpha-1}  \frac{\partial v_i}{\partial x_i}\quad\mbox{in $U_i$.}
$$ \noindent
Then, taking into account the positivity of $v_i,\, \forall \, i=1,..,N$,
\[ 
\int_{\Omega} \sum_{i=1}^{N} \bigg|  \frac{\partial \underline{u}}{\partial x_i} \bigg|^{p_{i}-2}  \frac{\partial \underline{u}}{\partial x_i} \, \frac{\partial \phi}{\partial x_i}  \, dx= \]
\[
\sum_{i=1}^{N} \int_{\prod_{j \ne i} U_{j} } \bigg{\{} \int_{U_i} \bigg[ \epsilon \alpha_{i} \, \bigg( \prod_{j\ne i} v_j^{\alpha_j} \bigg) v_i^{\alpha_i-1} \bigg]^{p_{i}-1} \bigg| \frac{\partial v_i}{\partial x_i}  \bigg|^{p_{i}-2} \frac{\partial v_i}{\partial x_i}  \, \frac{\partial \phi }{\partial x_i}  \, dx_{i} \bigg{\}} d\hat{x}^{i}
\]
with the obvious notation for $d\hat{x}^i$. Next, by using an integration by parts argument and the Fubini-Tonelli theorem, we get
\[
\int_{\Omega} \sum_{i=1}^{N} \bigg|\frac{\partial \underline{u}}{\partial x_i} \bigg|^{p_{i}-2} \frac{\partial \underline{u}}{\partial x_i} \, \frac{\partial \phi}{\partial x_i} \, dx= \]

\[
-\sum_{i=1}^{N} \int_{\prod_{j \ne i} U_j} \bigg( \epsilon \alpha_{i}  \prod_{j \ne i} v_j^{\alpha_j } \bigg)^{p_i -1}   \int_{U_i} \frac{\partial}{\partial x_i} \bigg( v_i^{(\alpha_i -1)(p_i -1)} \bigg|\frac{\partial v_i}{\partial x_i} \bigg|^{p_i -2} \frac{\partial v_i}{\partial x_i}  \bigg) \phi \, dx_i d\hat{x}^{i} \]

\[ + \sum_{i=1}^{N} \int_{\prod_{j \ne i} U_j} \bigg( \epsilon \alpha_{i}  \prod_{j \ne i} v_j^{\alpha_j } \bigg)^{p_i -1} \bigg{\{}  \int_{\partial U_i} v_i^{(\alpha_i -1)(p_i -1)} \bigg| \frac{\partial v_i}{\partial x_i} \bigg|^{p_{i}-2} \frac{\partial v_i}{\partial n_i} \phi \, dx_i \bigg{\}} d\hat{x}^{i}.
\]
The second term on the right can be discarded as $\frac{\partial v_i}{\partial {n_i}}<0 $ in $\partial U_i$, see (\ref{normal}). Considering that
$$
\frac{\partial}{\partial x_i} \bigg( v_i^{(\alpha_i -1)(p_i -1)} \bigg| \frac{\partial v_i}{\partial x_i} \bigg|^{p_i -2} \frac{\partial v_i}{\partial x_i}\bigg)=$$$$
(\alpha_i-1)(p_i-1) \bigg| \frac{\partial v_i}{\partial x_i} \bigg|^{p_i }+v_i^{(\alpha_i -1)(p_i -1)} \frac{\partial}{\partial x_i} \bigg( \bigg| \frac{\partial v_i}{\partial x_i} \bigg|^{p_i -2} \frac{\partial v_i}{\partial x_i} \bigg)
$$
and, from Lemma \ref{autovalor}, that
$$
\frac{\partial}{\partial x_i}\bigg( \bigg| \frac{\partial v_i}{\partial x_i} \bigg|^{p_i -2} \frac{\partial v_i}{\partial x_i} \bigg) =-\eta_i v_i^{p_i-1}\quad\mbox{in $U_i$,}
$$
our sub-solution condition becomes
\begin{equation} \label{CLAIM}
\begin{aligned}
 \sum_{i=1}^{N} \int_{\prod_{j \ne i} U_j}  &\bigg(\epsilon \alpha_i  \prod_{j \ne i} v_j^{\alpha_j}  \bigg)^{p_i -1} \int_{U_i} \bigg{\{} \bigg[ (1-\alpha_i)(p_i -1) v_i^{(\alpha_i -1)(p_i -1)-1} \bigg| \frac{\partial v_i}{\partial x_i} \bigg|^{p_{i}} +\\
& + v_{i}^{(\alpha_i -1)(p_i -1)} \eta_i v_i^{p_i -1} \bigg] - \lambda \epsilon^{q-1} \bigg( \prod_{k=1}^{N} v_k^{\alpha_k (q-1)} \bigg) \bigg{\}} \phi \, dx_i d\hat{x}^{i} \leq 0.
\end{aligned}
\end{equation} \noindent
Let us require a condition on the pointwise integrand
\[ \lambda \ge \]
\[
\sum_{i=1}^{N} \bigg(\epsilon \alpha_i  \prod_{j\ne i} v_j^{\alpha_j}  \bigg)^{p_i -q} v_{i}^{(\alpha_i -1)(p_i -1)-1- \alpha_i( q-1)}  \bigg[ (1-\alpha_i)(p_i-1)  \bigg| \frac{\partial v_i}{\partial x_i}  \bigg|^{p_{i}}  +  \eta_i  v_i^{p_{i}} \bigg]
 \]
 \[
= \sum_{i=1}^{N} \bigg(\epsilon \alpha_i  \prod_{j\ne i} v_j^{\alpha_j}  \bigg)^{p_i -q} v_{i}^{\alpha_i(p_i -q)-p_i}  \bigg[ (1-\alpha_i)(p_i-1)  \bigg| \frac{\partial v_i}{\partial x_i}  \bigg|^{p_{i}}  +  \eta_i  v_i^{p_{i}} \bigg].
 \]
Now we consider  various cases.
\begin{itemize}
\item If $1<q<p_{1}$, then by choosing $ \alpha_i>\frac{p_i}{(p_i -q)}>1$,
letting $\epsilon \rightarrow 0^{+}$ we obtain that $\underline{u}$ is a subsolution provided $\lambda>0$.

\item Assume that some $i_0 \in \{1,..,N\}$ we have $p_{i_0+1}>q\geq  p_{i _0}$ taking $p_{N+1}=\infty$. Then $\underline{u}$ is a subsolution if 
$$\lambda \ge \lambda_{*}:=\max_{U}{\mathcal S}$$
where 
$$
{\mathcal S}=\sum_{i=1}^N \bigg(\epsilon \alpha_i  \prod_{j\ne i} v_j^{\alpha_j}  \bigg)^{p_i -q} v_{i}^{\alpha_i (p_i -q)- p_i}  \bigg[ (1-\alpha_i)(p_i-1)  \bigg| \frac{\partial v_i}{\partial x_i} \bigg|^{p_{i}} + \eta_i  v_i^{p_{i}} \bigg]. 
$$ 
We show that $\lambda_{*}$ is finite. Observe that ${\mathcal S}={\mathcal S}_0+{\mathcal S}_1$ where
$$
{\mathcal S}_0=\sum_{i=1}^{i_0}\bigg(\epsilon \alpha_i  \prod_{j\ne i} v_j^{\alpha_j}  \bigg)^{p_i -q} v_{i}^{\alpha_i (p_i -q)- p_i}  \bigg[ (1-\alpha_i)(p_i-1)  \bigg| \frac{\partial v_i}{\partial x_i} \bigg|^{p_{i}} + \eta_i  v_i^{p_{i}} \bigg]. 
$$ 
and
$$
{\mathcal S}_1=\sum_{i=i_0+1}^N \bigg(\epsilon \alpha_i  \prod_{j\ne i} v_j^{\alpha_j}  \bigg)^{p_i -q} v_{i}^{\alpha_i (p_i -q)- p_i}  \bigg[ (1-\alpha_i)(p_i-1)  \bigg| \frac{\partial v_i}{\partial x_i} \bigg|^{p_{i}} + \eta_i  v_i^{p_{i}} \bigg]. 
$$ 

It is clear that ${\mathcal S}_1$ is finite, taking $\alpha_i>p_i/(p_i-q)$.\newline
On the other hand, observe that the behaviour next to $\partial U$ is controlled: when $v_i \rightarrow 0^+$ then as $\frac{\partial}{\partial n_i} v_i <0 $ on $\partial U_i$ we have that there exists $\delta>0$ small enough such that the quantity
$$ \bigg[ (1-\alpha_i)(p_i-1)  \bigg| \frac{\partial v_i}{\partial x_i} \bigg|^{p_{i}} +  \eta_i  v_i^{p_{i}} \bigg] <0\quad\mbox{in $U_i^\delta$,  for all $i=1,\ldots,p_{i_0}$},
$$
where
$$
U_{i}^\delta:= \{ x_i \in U_i: \dist(x_i \, , \partial U_i) \ge \delta\}.
$$
Moreover, ${\mathcal{S}}_0$ is bounded in $U \cap U_i^\delta$. Then,  ${\mathcal{S}}_0$ is bounded in $U$ and we can conclude that $\lambda_*$ is finite. 
\end{itemize}

\subsection{Supersolutions}
Since $\Omega$ is bounded, we can choose a domain $U$ such that 
$$ 
\Omega \subset U =\prod_{i=1}^{N} U^i, \qquad U^i=(a_i,b_i), \quad a_i,b_i \quad \text{in} \quad  \mathbb{R}.
$$ 
Now for $M>0$ we consider the function
$$
\overline{u}(x):= M \prod_{i=1}^{N} v_i(x_i), \quad \quad x \in \Omega, 
$$
where $v_{i}$ are the first eigenfunctions to the $p_i$-Laplacian in $U^{i}$, whose first eigenvalue we denote by $\eta^i$. Observe that 
$$\overline{u}_{\partial \Omega} >0.
$$
Then, $\overline{u}$ is a supersolution to (\ref{intro}) for all $ 0 \leq \phi \in W^{1,\bf{p}}_0(\Omega)$ holds 
$$
\lambda \int_{\Omega} M^{q-1} \prod_{i=1}^{N} v_{i}^{q-1} \phi \, dx \leq 
\int_{\Omega} \sum_{i=1}^{N}  \bigg| \frac{\partial \overline{u}}{\partial x_i} \bigg|^{p_i- 2}\frac{\partial \overline{u}}{\partial x_i} \, \frac{\partial \phi }{\partial x_i} \, dx.
$$
It is clear that
\[
\int_{\Omega} \sum_{i=1}^{N}  \bigg| \frac{\partial \overline{u}}{\partial x_i}  \bigg|^{p_i -2} \frac{\partial \overline{u}}{\partial x_i} \, \frac{\partial \phi }{\partial x_i} \, dx=\]
\[
\int_{\Omega}\sum_{i=1}^{N}\bigg[M \, \bigg( \prod_{j\ne i} v_j \bigg)  \bigg]^{p_{i}-1} \bigg| \frac{\partial v_i}{\partial x_i} \bigg|^{p_{i}-2} \frac{\partial v_i}{\partial x_i}\, \frac{\partial \phi }{\partial x_i} \, dx=\]
\[
-\sum_{i=1}^{N} \int_{\Omega} \bigg(M  \prod_{j \ne i} v_j^{} \bigg)^{p_i -1} \frac{\partial}{\partial x_i} \bigg(  \bigg| \frac{\partial v_i}{\partial x_i} \bigg|^{p_i -2} \frac{\partial v_i}{\partial x_i} \bigg) \phi \, dx=\]
\[
\sum_{i=1}^{N} \eta^{i}\int_{\Omega} \bigg(M   \prod_{j =1}^{N} v_j \bigg)^{p_i -1} \phi \, dx \]
Thus we may ask for the strong condition
\begin{equation} \label{lambdasup}
\lambda^*:= \sum_{i=1}^{N} \eta^i \bigg(M  \prod_{j =1}^N v_j^{} \bigg)^{p_i -q} \ge \lambda.
\end{equation}
Hence, if $ 1 <q < p_{N}$ by letting $M\rightarrow \infty$ we have that $\overline{u}$ is a super solution $\forall \lambda >0$.

\begin{proof}[Proof of Theorem \ref{main}] \quad 
\begin{enumerate}
\item Assume $1<q<p_1$. Fix $\lambda>0$. Then, we can choose $\epsilon>0$ small and $M$ large enough such that $\underline{u}$, $\overline{u}$ are sub-supersolution of (\ref{intro}) and $\underline{u}\leq \overline{u}$ in $\Omega$. Theorem \ref{teoss} assures the existence of a solution $u$ of (\ref{intro}) such that $\underline{u}\leq u\leq \overline{u}$. This completes this case. 
\item Assume $p_1\leq q <p_N$. In this case, taking for example $\epsilon=1$, we have that $\underline{u}$ is subsolution provided that $\lambda\geq\lambda_*$ for some $\lambda_*$. On the other hand, we can take $M$ large such that $\overline{u}$ is supersolution and $\underline{u}\leq \overline{u}$. Thus, there exists a positive solution for $\lambda \geq \lambda_*$.\newline
Now, we define
$$
\Lambda:= \inf \{\lambda: \mbox{(\ref{intro}) possesses at least a positive solution}\}.
$$
We have proved that $\Lambda<\infty$. For $p_1<q<p_N$, in \cite{ Monte} it was proved that $0<\Lambda$. We show now that this is also true for $q=p_1$. Indeed, let now consider $q=p_1$ and let us multiply the equation \eqref{intro} by $u$ and integrate it on $\Omega$ to obtain
\[
\sum_{i=1}^N \int_{\Omega} \bigg| \frac{\partial u}{\partial x_i} \bigg|^{p_i} dx = \sum_{i=1}^N \bigg| \bigg| \frac{\partial u}{\partial x_i} \bigg| \bigg|^{p_i}_{p_i}= \lambda ||u||_{p_1}^{p_1}= \lambda \int_{\Omega} |u|^{p_1} dx
\] 
Now we use the embedding (\ref{embedding}) on $r=p_1$ to get
\[
\bigg( \frac{d^1 p_1}{2} \bigg)^{-p_1} ||u||_{p_1}^{p_1} \leq \bigg| \bigg| \frac{\partial u}{\partial x_1} \bigg| \bigg|^{p_1}_{p_1}+ \sum_{i=2}^N \bigg| \bigg| \frac{\partial u}{\partial x_i} \bigg| \bigg|^{p_i}_{p_i}= \lambda ||u||_{p_1}^{p_1}
\] and thus
\[
||u||_{p_1}^{p_1} \bigg[ \lambda - \bigg( \frac{2}{d^1 p_1} \bigg)^{p_1} \bigg] \geq 0.
\] 
But if $\lambda< \bigg( \frac{2}{d^1 p_1} \bigg)^{p_1} $ this quantity is negative and we arrive to the absurd of declaring $||u||_{p_1} \ne 0$.\newline
We prove now that for all $\lambda>\Lambda$ we have the existence of positive solution. Indeed, fix $\lambda_0>\Lambda$. Then, by definition of $\Lambda$, there exists $\mu\in(\Lambda,\lambda_0)$ and a positive solution, denoted by $u_\mu$, of (\ref{intro}) for $\lambda=\mu$. Since $\mu<\lambda_0$, it is clear that $u_\mu$ is subsolution of (\ref{intro}) for $\lambda=\lambda_0$. On the other hand, for $M$ large, there exists $\overline{u}$ supersolution of (\ref{intro}) for $\lambda=\lambda_0$. Finally, thanks to regularity results, see for instance Proposition 4.1 in \cite{ElHam} or Lemma 2.4 in \cite{giosusbsup}, we have that 
$u_\mu\in L^\infty(\Omega)$. Hence, for $M$ large $u_\mu\leq \overline{u}$, and we can conclude the existence of positive solution for $\lambda=\lambda_0$. This completes the proof.
\end{enumerate}
\end{proof}

\begin{remark}
Since our subsolution $\underline{u}$ is strictly positive in $U$, we have by Theorem \ref{teoss} that $u \ge \underline{u}>0$ in a non empty open set contained in $\Omega$. In the case $p_1\ge 2$ by the result of Corollary 4.4 \cite{Monte} we have $u>0$ in $\Omega$. 
\end{remark} 
    
    \begin{remark} We comment a possible further generalization.\newline
Let $ x=(x_{1},...,x_N)$, where $x_i \in \Omega_i \subset \mathbb{R}^{N_i}$ , being $\Omega_i$ an open, bounded and convex domain. Denote with $\nabla_{x_i}$ the gradient along the vector $x_i$ and $\text{div}_{x_i}$ its divergence, and let 
$$
\Delta_{p_i} u= {\text{div}}_{x_i} (|\nabla_{x_i} u|^{p_i-2}\nabla_{x_i} u)= \sum_{j=1}^{N_i}\frac{\partial}{\partial x_{ij}}\left( \left|\frac{\partial u}{\partial x_{ij}}\right|^{p_i-2}\frac{\partial u}{\partial x_{ij}} \right)
$$ 
be the $p_i$-Laplacian acting on the $x_i$ vector. Problems of the kind of

\begin{equation}
\label{generalisation}
\left\{\begin{array}{ll}
-\displaystyle \sum_{i=1}^N  \Delta_{p_i} u =\lambda u^{q-1}  &\mbox{in $\Omega= \prod {\Omega_i}$,}\\
u=0 & \mbox{on $\partial\Omega$,}
\end{array}
\right.
\end{equation} \noindent 
can be faced with the same technique, using properties of the $p_i$-Laplacian principal eigenfuctions, and owing integrability condition as \eqref{cosa} to recent regularity results obtained in \cite{Cia}. 
\end{remark}

\section*{Aknowledgments}

The authors wish to thank professors A. Cianchi and V. Vespri for enlighting conversations about the problem. Moreover, the authors acknowledge IMUS (Mathematics Institute of the University of Seville), IEMath-GR (Mathematics Institute of the University of Granada) and University of Cadiz for supporting a PhD course from which the collaboration arose. S. Ciani is partially founded by INdAM group GNAMPA, A. Su\'arez  by PGC2018-098308-B-I00 (MCI/AEI/FEDER, UE) and G. M. Figueiredo by CNPQ, CAPES and FAP-DF.


\end{document}